\documentclass[reqno,12pt]{amsart}
\usepackage[left=1.2in, right = 1.2in, top=1in]{geometry}
\usepackage{amsmath, amssymb, amsthm, mathrsfs, color, times, textcomp, verbatim}
\usepackage{amstext}
\usepackage{appendix}
\usepackage{verbatim}
\usepackage{cases}

\usepackage{graphicx}
\usepackage{tikz}
\newtheorem{theorem}{Theorem}[section]

\newtheorem{corollary}{Corollary}[section]

\newtheorem{lemma}{Lemma}[section]

\allowdisplaybreaks[4]
\usepackage{fancyhdr}
\usepackage{hyperref}

\title{$\sigma_k$-Yamabe measure}

\author{Xi-Nan Ma}
\address{Department of Mathematics \\University of Science and Technology of China, Hefei, China}
\email{xinan@ustc.edu.cn}
\author{Wangzhe Wu}
\address{Institute of Mathematics \\Academy of Mathematics and Systems Science, Chinese
Academy of Sciences, Beijing, 100190, China}
\email{wuwz18@mail.ustc.edu.cn}

\subjclass[2020]{58C35, 28A33, 35J60, }
\keywords{$\sigma_k$-Yamabe problem, Hessian measure}
\begin{document}

	\pagestyle{fancy}

\fancyhead{}
\fancyhead[CO]{$\sigma_k$-Yamabe measure}

\fancyhead[CE]{\leftmark}

	\begin{abstract}
		We found a special divergence structure for the $\sigma_k$-Yamabe operator and use it to get a monotonicity formula.  We also get an interior $L^{\infty}$ estimate via its $L^1$ norm for the $\sigma_k$-Yamabe operator when $1\le k \le \frac{n}{2}$.		
		Combining these two tools, we prove the weak continuity of the $\sigma_k$-Yamabe measure with respect to convergence in measure.
	\end{abstract}

	\maketitle

	\section{Introduction}
	
	Let $(M, g)$ be a Riemannian manifold of dimension $n \geq 3$ with a metric $g$. The well-known Yamabe problem states whether there exists metrics which are pointwise 
	conformal to g and have constant scalar curvature. The Yamabe problem was solved  through the work of Yamabe\cite{MR0125546} , Trudinger\cite{MR0240748} , Aubin\cite{MR0431287} and Schoen\cite{MR0788292}. Denote $Ric$ and $R$ as the Ricci tensor and the scalar curvature, respectively. Then the Schouten tensor is
	\begin{equation}
		\tilde{A}^{g}_{ij} = \frac{1}{n - 2}\left[ Ric^{g}_{ij} - \frac{1}{2(n - 1)}R^g g_{ij} \right].
	\end{equation}
	Now transform the $(0, 2)$-tensor $\tilde{A}^{g}_{ij}$ to a $(1, 1)$-tensor $A_{ij}$ by $A^g = g^{-1}\tilde{A}^g$. 
	
	We are always interested in lcf manifolds. We say a Riemannian manifold $(M, g)$ is lcf if the metric $g$ can be locally written as $g = v^{-2}|dx|^2$ for some smooth $v > 0$, where $|dx|^2$ is the usual Euclidean metric. Then the $(1, 1)$-tensor becomes
	\begin{equation}
		A_{i j}(v) = v v_{ij} - \frac{1}{2}|\nabla v|^{2}\delta_{ij}.
	\end{equation}

For $\lambda = (\lambda_1,\cdots, \lambda_n) \in \mathbb R^n$, we define 	\begin{align}\sigma_k(\lambda) :=  \sum_{1 \leq i_1 < \cdots < i_{k} \leq n} \lambda_{i_1} \cdots \lambda_{i_k}.\end{align}	

Let $ \lambda \{A_{ij}(v)\} $ are the eigenvalues
of $A_{i j}(v)$, then we define
\begin{align}
	S_{k}(v): = \sigma_{k}(\lambda \{A_{ij}(v)\}).
\end{align}	

	We call $S_k$ the $\sigma_k$-Yamabe operator. With respect to this operator and the related $\sigma_k$-Yamabe problem, there has been a lot of work, for instance, Chang- Gursky- Yang \cite{MR1945280, MR1923964}, Gursky- Viaclovsky \cite{MR2373147}, Li- Li \cite{MR1988895}, Guan- Wang \cite{MR1978409}, Sheng- Trudinger- Wang \cite{MR2362323}, Ge-Wang\cite{MR2290138}.

	 For the Hessian measure, Trudinger - Wang in \cite{MR1634570} and \cite{MR1726702} introduced the notion of $k$-Hessian measure associated with $k$-convex functions and proved the weak continuity of the associated $k$-Hessian measure with respect to convergence in measure. In \cite{MR1923626},  they gave some application of the $k$-Hessian measure, especially to the Dirichlet problem. 
	In Dai-Trudinger-Wang \cite{MR2911884} and Dai-Wang-Zhou \cite{MR3436399}, they
	generalized the Hessian measure result to higher order curvature operator with respect to almost everywhere convergence.  For the Hessian measure on Heisenberg group, Trudinger-Zhang \cite{MR3035058} got the locally uniformly 
	convergence also through the monotonicity formula for $k$-convex functions on Heisenberg group.

	  In this paper, we will introduce the $\sigma_k$-Yamabe measure and  proved the weak continuity of the associated  $\sigma_k$-Yamabe measure with respect to convergence in measure. 
	
	We recall that the Garding's cone is defined as
	\begin{equation}
		\Gamma_{k} = \{ \lambda \in \mathbb R^n:\sigma_{i}(\lambda) \geq 0, \forall 1 \leq i \leq k \}.
	\end{equation} 
	Define 
	\begin{equation}\label{7.10.21}
		\begin{aligned}
			S^{ij}_{k}&:= \frac{\partial S_{k}}{\partial A_{ij}},\\
			\tilde{\Phi}^k(\Omega) &:= \Big\{ u \in C^2(\Omega): u > 0, \lambda \{A_{ij}(v)\} \in \Gamma_k  \},\\
			\Phi^k(\Omega) &:= \Big\{ u \in L^1_{loc}(\Omega): u > 0, \text{and there exists a sequence $ \{ u^{(m)} \} \in \tilde\Phi^k(\Omega)$ },\\
			& \text{ such that $u^{(m)} $ converges to $u$ in $L^1_{loc}$} \Big\}.
		\end{aligned}
	\end{equation}
Corresponding to Theorem 1.1 in  Trudinger-Wang  \cite{MR1726702}, we shall establish the following
characterization of  $\sigma_k$-Yamabe measure on $\Phi^k(\Omega)$.
		\begin{theorem}\label{mainresult6.18}
		For any $u \in \Phi^k(\Omega)$, there exists a Borel measure $\mu_{k}[u]$ in $\Omega$ such that
		\begin{itemize}
			\item $\mu_{k}[u] = S_{k} (u)$ for $u \in C^2(\Omega)$.
			\item If $\{ u^{(m)} \}$ is a sequence in $\Phi^k(\Omega)$ converging locally in measure to a function $u \in \Phi^k(\Omega)$, then the corresponding measure $\mu_{k}[u^{(m)}] \rightarrow \mu_{k}[u]$ weakly.
		\end{itemize}
		
	\end{theorem}
In order to prove the above theorem,  we introduce a monotonicity formula with respect to $\sigma_k$-Yamabe operator:
		\begin{theorem}\label{5.6monotonicity}
		Let $u, v \in \Phi^2(\Omega) \cap C^2(\bar{\Omega})$ satisfy $u \geq v$ in $\Omega$ and $u = v$ on $\partial \Omega$. Then for any $\alpha \geq 0$, it follows that :
		\begin{equation}
			\int_{\Omega}  \sum_{l = 0}^k b_{l} u^{\alpha}|\nabla u|^{2l}S_{k - l}(u)\leq  \int_{\Omega}  \sum_{l = 0}^k b_{l} v^{\alpha}|\nabla v|^{2l}S_{k - l}(v),
		\end{equation}
		with
		\begin{equation}
			\begin{aligned}
				b_{l }&=  \Big[(k + l)\alpha + kn + 2kl \Big]\cdot \frac{(\alpha + n + l - 1)!}{2^{l}l! (\alpha + n)!} \geq 0, \text{ for $l = 0,\cdots, k$}.  
			\end{aligned}
		\end{equation}
	\end{theorem}
	
	This  monotonicity formula is the most important tool to obtain  the proof of the Theorem 1.1.  It is different from  the usual monotonicity formulas on $k$-Hessian measure  in Trudinger-Wang in \cite{MR1634570}, Dai-Wang-Zhou \cite{MR3436399}. In some sense, it is similar to the monotonicity formula for $k$-convex functions on Heisenberg group by Trudinger-Zhang \cite{MR3035058}.
	
	The plan of this paper is as follows. In the next section we give various
	properties of $\sigma_k$-Yamabe operator, especially we obtain the relation between $S_{m}(v)$ and $v^m\sigma_{m}(\lambda \{v_{ij}\})$ in Lemma 2.3. In Section 3, we
first	use the special divergence structure with respect to the $\sigma_k$-Yamabe operator to get a differential identity  in Lemma 3.1, which will be also used to get the upper bound estimates via Moser iteration in Section 5.  Then we prove the monotonicity formulas i.e. the  Theorem 1.2. The new idea on this step is that we introduce the combination of $\psi(v) = \psi^{\alpha}_{k}(v):= \sum_{l = 0}^k b_{l} w^{\alpha}|\nabla v|^{2l}S_{k - l}(v)$ to replace the usual term $\sigma_{m}(\lambda \{v_{ij}\})$ as in Trudinger - Wang in \cite{MR1634570}.   Then we give comparison principle as a consequence of the monotonicity formulas.
	In Section 4, we complete the proof of Theorem 1.1 for $k> \frac{n}{2}$. In this case we use the Holder estimate Theorem 2.7, integral estimatese Theorem 3.1 in  Trudinger-Wang in \cite{MR1634570}  with the help from Lemma 2.3 in section 2, and we use the monotonicity formulas  Theorem 1.2.  The proof of the weak continuity result, Theorem 1.1 is then completed for $k> \frac{n}{2}$. 	
	In Section 5, we first obtain  the  interior $L^{\infty}$ bound with respect to the $\sigma_k$-Yamabe operator for $1\le k \le \frac{n}{2}$ via Moser iteration. As a consequence we   get a local uniform integral estimates from the differential identity  Lemma 3.1.	
	 Finally, in Section 6,   we use the integral estimates to complete the proof the weak continuity result, Theorem 1.1 for $1\le k\le \frac{n}{2}$. In this step we also follow that idea from Trudinger-Wang in \cite{MR1634570}.  
	 
	\textbf{Acknowledgment:} The first author would like to thank the helpful
	discussion and encouragement from Prof. X.-J. Wang on this subject.  The authors were supported by the National Natural Science Foundation of China (grants 12141105) and the first author also was supported by the National Key Research and Development Project (grants SQ2020YFA070080).

	\section{Properties of $\sigma_k$-Yamabe operator}
	In this section, we will give some fundamental properties of $\sigma_k$-Yamabe operator which will be widely used in this paper.  The first lemma appeared in Li-Nguyen-Wang
\cite{MR4691488}, for completely we contain its proof.  In this paper we always use the notation $S_{k}:= S_{k}(v): = \sigma_{k}(\lambda \{A_{ij}(v)\})$ and $\sigma_k:= \sigma_k(v):= \sigma_k(\lambda\{v_{ij}\})$.
	
	\begin{lemma}\cite{MR4691488}\label{5.7sum}
		If $u, v \in 	\tilde{\Phi}^k(\Omega)$, then $u + v \in 	\tilde{\Phi}^k(\Omega)$.
	\end{lemma}
	\begin{proof}[Proof of Lemma \ref{5.7sum}]
		Define $w := u + v$, then 
		\begin{align*}
			A_{i j}(w) &= (u + v)(u_{i j} + v_{i j}) - \frac{1}{2}|\nabla u + \nabla v|^2 \delta_{i j}\\
			&= (u + v)(u^{-1}A_{i j}(u) + v^{-1} A_{i j}(v)) + \frac{1}{2}(u + v)(u^{-1}|\nabla u|^2 \delta_{i j} + v^{-1}|\nabla v|^2 \delta_{i j})\\
			&- \frac{1}{2}|\nabla u + \nabla v|^2 \delta_{i j}\\
			&= (u + v)u^{-1}A_{i j}(u) + (u + v) v^{-1} A_{i j}(v) + \frac{1}{2}\delta_{i j}(v u^{-1}|\nabla u|^2 + uv^{-1}|\nabla v|^2 - 2\nabla u \cdot \nabla v).
		\end{align*}
		Since
		\begin{equation}
			v u^{-1}|\nabla u|^2 + uv^{-1}|\nabla v|^2 - 2\nabla u \cdot \nabla v \geq 0,
		\end{equation}
		and $\Gamma_k$ is convex , we get that 
		\begin{equation}
			\lambda(A_{i j}(w)) \in  \Gamma_k.
		\end{equation}
	\end{proof}
	
		By  Gonz\'alez \cite{MR2169873,  MR2247857, MR2263673}, we know that
	\begin{lemma}\label{2024.7.10lemma1}
		If $u \in 	\tilde{\Phi}^k(\Omega)$, then for any $1 \leq l \leq k$, the matrices $(S_{l}^{i j})$ defined in \eqref{7.10.21} is nonnegative-definite and
		\begin{equation}
			\sum_{j} \partial_{j}S^{ij}_{m + 1} = -(n - m)S_m v_i v^{-1} + n \sum_{j} S^{ij}_{m + 1}v_i v^{-1}.
		\end{equation}
	\end{lemma}
	Next we will show $\lambda(u_{ij})\in \Gamma_k$ if $\lambda(A_{ij}(u))\in \Gamma_k$.
	\begin{lemma}\label{5.6.sigma}
		If $v \in 	\tilde{\Phi}^k(\Omega)$, then there exists a positive constant $C$ only depends on $n,l,q$ such that
		\begin{equation}
			v^{l}\sigma_{l}(v) \geq C|\nabla v|^{2q} S_{l - q} , \text{ with $1 \leq q < l \leq k$},
		\end{equation}
		and
		\begin{equation}
			S_{m} \leq v^{m}\sigma_{m}, \text{ with $1 \leq m \leq k$}.
		\end{equation}
	\end{lemma}
	\begin{proof}[Proof of Lemma \ref{5.6.sigma} ]
			Let $\lambda = (\lambda_1, \cdots, \lambda_{n})$ be the eigenvalue vector of the matrix $( v_{i j})$ and $b := \frac{1}{2}|\nabla v|^2$. Define 
		\begin{equation}
			\sigma_k := \sigma_{k}(v_{i j}) = \sum_{1 \leq i_1 < \cdots < i_{k} \leq n} \lambda_{i_1} \cdots \lambda_{i_k}.
		\end{equation} Then
		\begin{equation}
			\begin{aligned}
				S_{k}(v) &= \sum_{1 \leq i_1 < \cdots < i_{k} \leq n} (v\lambda_{i_1} - b)\cdots (v\lambda_{i_k} - b)\\
				&= v^{k}\sigma_{k} +  \sum_{l = 1}^k \frac{C^l_k C^k_n}{C^{k - l}_n} v^{k - l}\sigma_{k - l} (-b)^{l}.
			\end{aligned}
		\end{equation}
		Suppose
		\begin{equation}
			S_{k}(v) = v^{k}\sigma_{k} + \sum_{l = 1}^k A_{l} S_{k - l} (-b)^{l},
		\end{equation}
		then
		\begin{align*}
			S_{k}(v) &= v^{k}\sigma_{k} + \sum_{m = 1}^k A_{m} S_{k - m} (-b)^{m}\\
			&= v^{k}\sigma_{k} + \sum_{m = 1}^k A_{m}  (-b)^{m}\Big[ v^{k - m}\sigma_{k - m} +  \sum_{l = 1}^{k - m} \frac{C^l_{k - m} C^{k - m}_n}{C^{k - m - l}_n} v^{k - m - l}\sigma_{k - m - l} (-b)^{l}  \Big]\\
			&= v^{k}\sigma_{k} + \sum_{m = 1}^k A_{m}  (-b)^{m} v^{k - m}\sigma_{k - m} + \sum_{m = 1}^k A_{m}   \sum_{l = 1}^{k - m} \frac{C^l_{k - m} C^{k - m}_n}{C^{k - m - l}_n} v^{k - m - l}\sigma_{k - m - l} (-b)^{m + l} \\
			&=v^{k} \sigma_{k} + \sum_{m = 1}^k A_{m}  (-b)^{m} v^{k - m}\sigma_{k - m} + \sum_{m = 2}^k (-b)^{m} v^{k - m}\sigma_{k - m} \sum_{l = 1}^{m - 1} A_{l} \frac{C^{m - l}_{k - l} C^{k - l}_n}{C^{k - m}_n}\\
			&= v^{k}\sigma_{k}  + \sum_{m = 1}^k (-b)^{m} v^{k - m}\sigma_{k - m} \sum_{l = 1}^{m} A_{l} \frac{C^{m - l}_{k - l} C^{k - l}_n}{C^{k - m}_n}.
		\end{align*}	
		So we get that for any $1 \leq m \leq k$
		\begin{align*}
			& \sum_{l = 1}^{m } A_{l} C^{m - l}_{k - l} C^{k - l}_n  = C^m_k C^k_n.
		\end{align*}
		
		\begin{equation}\label{5.5.equ}
			\Rightarrow\sum_{l = 1}^{m }\frac{m! (n - k)!}{(m - l)! (n - k + l)!} A_{l} = 1.
		\end{equation}
		We find the solutions to \eqref{5.5.equ} are
		\begin{equation}
			A_{l} = (-1)^{l - 1}\frac{(n - k + l)!}{l! (n - k)!}, \text{ with $1 \leq l \leq k$.}
		\end{equation}
		Thus we get
		\begin{equation}\label{5.6.equ1}
			\begin{aligned}
				S_{k}(v) &= v^{k}\sigma_{k} - \sum_{l = 1}^k \frac{(n - k + l)!}{l! (n - k)!}S_{k - l} b^{l}.
			\end{aligned}
		\end{equation}

	\end{proof}
	So by Lemma \ref{5.6.sigma}, we know that if $u \in \tilde\Phi^{k}(\Omega)$, then $\sigma_{l}(u_{i j}) \geq 0$ for any $0 \leq l \leq k$.

	\section{Monotonicity formula}

	In this section, inspired by Lemma 3.3 in Gonz\'alez \cite{MR2169873}, we give a special divergence structure with respect to the $\sigma_k$-Yamabe operator.

	\begin{lemma}\label{6.19divergence}
		There exist $\{a_{l}\}_{l = 0}^{k - 1}$ and $\{b_{l}\}_{l = 0}^{k}$, such that
		\begin{equation}
		\sum_{l = 0}^{k - 1} a_{l} (v^{\alpha + 1}|\nabla v|^{2l}S^{i j}_{k - l} v_{i})_{j} +\sum_{l = 0}^k b_{l} v^{\alpha}|\nabla v|^{2l}S_{k - l} = 0,
	\end{equation}
	with $a_0 = - 1.$
	\end{lemma}
	\begin{proof}  We use the computation from  Gonz\'alez \cite{MR2169873}.
		\begin{equation}
		\begin{aligned}
			&(k - l)v^{\alpha}|\nabla v|^{2l}S_{k - l} \\
			=&  v^{\alpha}|\nabla v|^{2l}S^{i j}_{k - l}A_{i j}\\
			=&  v^{\alpha}|\nabla v|^{2l}S^{i j}_{k - l}\left(v v_{i j} - \frac{1}{2}|\nabla v|^2 \delta_{i j}\right)\\
			=& (v^{\alpha + 1}|\nabla v|^{2l}S^{i j}_{k - l} v_{i})_{j} - (\alpha + 1) v^{\alpha}|\nabla v|^{2l}S^{i j}_{k - l} v_{i}v_{j} - 2l v^{\alpha + 1}|\nabla v|^{2l - 2}S^{i j}_{k - l} v_{i} v_{m j} v_{m}\\
			&- v^{\alpha + 1}|\nabla v|^{2l} v_{i} \cdot \partial_{j}S^{i j}_{k - l} - \frac{1}{2}(n + l + 1 - k) v^{\alpha}|\nabla v|^{2l + 2} S_{k - l - 1}\\
			=& (v^{\alpha + 1}|\nabla v|^{2l}S^{i j}_{k - l} v_{i})_{j} - (\alpha + 1) v^{\alpha}|\nabla v|^{2l}S^{i j}_{k - l} v_{i}v_{j} - 2l v^{\alpha}|\nabla v|^{2l - 2}S^{i j}_{k - l} v_{i} v_{m} \left( A_{m j} +  \frac{1}{2}|\nabla v|^2 \delta_{m j}\right)\\
			&- v^{\alpha + 1}|\nabla v|^{2l} v_{i} \cdot \Big[ -(n - k + l + 1) v^{-1} S_{k - l- 1}v_{i} + n v^{-1} S^{i j}_{k - l}v_{j} \Big] \\
			&- \frac{1}{2}(n + l + 1 - k) v^{\alpha}|\nabla v|^{2l + 2} S_{k - l - 1},
		\end{aligned}
	\end{equation}
	
	\begin{align*}
		&\Rightarrow (k - l)v^{\alpha}|\nabla v|^{2l}S_{k - l} \\
			=& (v^{\alpha + 1}|\nabla v|^{2l}S^{i j}_{k - l} v_{i})_{j} - (\alpha + n + l + 1) v^{\alpha}|\nabla v|^{2l}S^{i j}_{k - l} v_{i}v_{j} - 2l v^{\alpha}|\nabla v|^{2l - 2}S^{i j}_{k - l} v_{i} v_{m}  A_{m j} \\
			&  + \frac{1}{2}(n + l + 1 - k) v^{\alpha}|\nabla v|^{2l + 2} S_{k - l - 1}\\
			=& (v^{\alpha + 1}|\nabla v|^{2l}S^{i j}_{k - l} v_{i})_{j} - (\alpha + n + l + 1) v^{\alpha}|\nabla v|^{2l}S^{i j}_{k - l} v_{i}v_{j} - 2l v^{\alpha}|\nabla v|^{2l - 2} v_{i} v_{m} (S_{k - l}\delta_{i m} - S^{i m}_{k - l + 1})  \\
			&  + \frac{1}{2}(n + l + 1 - k) v^{\alpha}|\nabla v|^{2l + 2} S_{k - l - 1},\\
	\end{align*}

	\begin{equation}\label{4.28.828}
		\begin{aligned}
			&\Rightarrow (k + l)v^{\alpha}|\nabla v|^{2l}S_{k - l} \\
			=& (v^{\alpha + 1}|\nabla v|^{2l}S^{i j}_{k - l} v_{i})_{j} - (\alpha + n + l + 1) v^{\alpha}|\nabla v|^{2l}S^{i j}_{k - l} v_{i}v_{j} + 2l v^{\alpha}|\nabla v|^{2l - 2}  S^{i j}_{k - l + 1} v_{i} v_{j} \\
			&  + \frac{1}{2}(n + l + 1 - k) v^{\alpha}|\nabla v|^{2l + 2} S_{k - l - 1}.\\
		\end{aligned}
	\end{equation}
	Write $S_{-1} = 0$, then we get that when $0 \leq l \leq k$,
	\begin{equation}\label{4.28}
		\begin{aligned}
			&\Rightarrow (v^{\alpha + 1}|\nabla v|^{2l}S^{i j}_{k - l} v_{i})_{j} \\
			=&(k + l)v^{\alpha}|\nabla v|^{2l}S_{k - l}  +  (\alpha + n + l + 1) v^{\alpha}|\nabla v|^{2l}S^{i j}_{k - l} v_{i}v_{j} - 2l v^{\alpha}|\nabla v|^{2l - 2}  S^{i j}_{k - l + 1} v_{i} v_{j} \\
			&- \frac{1}{2}(n + l + 1 - k) v^{\alpha}|\nabla v|^{2l + 2} S_{k - l - 1}.\\
		\end{aligned}
	\end{equation}
	We hope that 
	\begin{equation}
		\sum_{l = 0}^{k - 1} a_{l} (v^{\alpha + 1}|\nabla v|^{2l}S^{i j}_{k - l} v_{i})_{j} +\sum_{l = 0}^k b_{l} v^{\alpha}|\nabla v|^{2l}S_{k - l} = 0.
	\end{equation}
	Using \eqref{4.28}, this means that
	\begin{align*}
		0&= \sum_{l = 0}^{k - 1} a_{l} \Big[ (k + l)v^{\alpha}|\nabla v|^{2l}S_{k - l}  +  (\alpha + n + l + 1) v^{\alpha}|\nabla v|^{2l}S^{i j}_{k - l} v_{i}v_{j} - 2l v^{\alpha}|\nabla v|^{2l - 2}  S^{i j}_{k - l + 1} v_{i} v_{j} \\
			&- \frac{1}{2}(n + l + 1 - k) v^{\alpha}|\nabla v|^{2l + 2} S_{k - l - 1} \Big]\ + \sum_{l = 0}^{k} b_{l} v^{\alpha}|\nabla v|^{2l}S_{k - l} \\
		&= b_{k}v^{\alpha}|\nabla v|^{2k} + \sum_{l = 0}^{k - 1} \Big[ a_{l} (k + l) + b_{l}\Big] v^{\alpha}|\nabla v|^{2l}S_{k - l} - \sum_{l = 0}^{k - 1}\frac{1}{2}a_{l}(n + l + 1 - k) v^{\alpha}|\nabla v|^{2l + 2} S_{k - l - 1} \\
		&+ \sum_{l = 0}^{k - 1} a_{l} (\alpha + n + l + 1) v^{\alpha}|\nabla v|^{2l}S^{i j}_{k - l} v_{i}v_{j} - \sum_{l = 0}^{k - 1} 2l a_{l} v^{\alpha}|\nabla v|^{2l - 2}  S^{i j}_{k - l + 1} v_{i} v_{j}\\
		&= b_{k}v^{\alpha}|\nabla v|^{2k} +  (a_{0} k + b_{0}) v^{\alpha}S_{k} - \frac{1}{2}n a_{k - 1} v^{\alpha}|\nabla v|^{2k} + a_{k - 1} (\alpha + n + k) v^{\alpha}|\nabla v|^{2k}  \\
		&+ \sum_{l = 1}^{k - 1} \Big[ a_{l} (k + l) + b_{l}\Big] v^{\alpha}|\nabla v|^{2l}S_{k - l} - \sum_{l = 0}^{k - 2}\frac{1}{2}a_{l}(n + l + 1 - k) v^{\alpha}|\nabla v|^{2l + 2} S_{k - l - 1} \\
		&+ \sum_{l = 0}^{k - 2} a_{l} (\alpha + n + l + 1) v^{\alpha}|\nabla v|^{2l}S^{i j}_{k - l} v_{i}v_{j} - \sum_{l = 1}^{k - 1} 2l a_{l} v^{\alpha}|\nabla v|^{2l - 2}  S^{i j}_{k - l + 1} v_{i} v_{j}\\
		&= \Big[ b_{k} + a_{k - 1} (\alpha + \frac{1}{2} n + k)  \Big] v^{\alpha}|\nabla v|^{2k} +  (a_{0} k + b_{0}) v^{\alpha}S_{k}   \\
		&+ \sum_{l = 1}^{k - 1} \Big[ a_{l} (k + l) + b_{l}\Big] v^{\alpha}|\nabla v|^{2l}S_{k - l} - \sum_{l = 1}^{k - 1}\frac{1}{2}a_{l - 1}(n + l - k) v^{\alpha}|\nabla v|^{2l} S_{k - l} \\
		&+ \sum_{l = 1}^{k - 1} a_{l - 1} (\alpha + n + l) v^{\alpha}|\nabla v|^{2l - 2}S^{i j}_{k - l + 1} v_{i}v_{j} - \sum_{l = 1}^{k - 1} 2l a_{l} v^{\alpha}|\nabla v|^{2l - 2}  S^{i j}_{k - l + 1} v_{i} v_{j}.\\
	\end{align*}
	Therefore we get that 
	\begin{numcases}{}
		 \label{4.28.1} b_{k} + a_{k - 1} (\alpha + \frac{1}{2} n + k) = 0,\\
		 \label{4.28.2}a_{0} k + b_{0} = 0,\\
		\label{4.28.3} a_{l} (k + l) + b_{l} - \frac{1}{2}a_{l - 1}(n + l - k) = 0,\\
		\label{4.28.4} a_{l - 1} (\alpha + n + l)  - 2l a_{l} = 0,
	\end{numcases}
	with $1 \leq l \leq k - 1$.  From now on, we always let $a_0 = -1, b_{0} = k$. So for fixed $n, k, \alpha$, the above equations always have a unique solution $\{ a_{0}, \cdots , a_{k - 1}, b_{0}, \cdots, b_{k - 1}, b_{k}\}$:
	
	\begin{equation}
		a_{l} = - \frac{(\alpha + n + l)!}{2^{l}\cdot l!(\alpha + n)!}, \text{ for $l = 0, 1, \cdots, k - 1$,}
	\end{equation}
	and
	\begin{equation}\label{6.19b}
			\begin{aligned}
				b_{l + 1} &=  \Big[(k + l + 1)\alpha + kn + 2k(l + 1) \Big]\cdot \frac{(\alpha + n + l)!}{2^{l + 1}(l + 1)! (\alpha + n)!}. 
			\end{aligned}
		\end{equation}
		with $0 \leq l \leq k - 1$ and $b_{0} = k.$
	\end{proof}
	From now on, let us define
	\begin{equation}\label{5.6psi}
		\psi(w) = \psi^{\alpha}_{k}(w):= \sum_{l = 0}^k b_{l} w^{\alpha}|\nabla w|^{2l}S_{k - l}(w).
	\end{equation}

	\begin{lemma}\label{6.18lemma2}
		For  any $u, v \in \Phi^k(\Omega)\cap C^2(\Omega)$, we have 
		\begin{equation}
			\begin{aligned}
				\frac{\partial}{\partial t}\psi &= \sum_{l = 0}^{k}  b_{l}\Big[w^{\alpha + 1}|\nabla w|^{2l}S_{k - l}^{i j}  (v - u)_{i}\Big]_{j}  +  2 \sum_{l = 0}^{k} l b_{l}  \Big[w^{\alpha}|\nabla w|^{2l - 2}S^{i j}_{k - l + 1} w_{i}(v - u) \Big]_{j}\\
				& -\sum_{l = 0}^{k} (\alpha + n + l + 1) b_{l} \Big[w^{\alpha}|\nabla w|^{2l}S^{i j}_{k - l} w_{i}(v - u) \Big]_{j}.\\
			\end{aligned}
		\end{equation}
	\end{lemma}
	\begin{proof}[Proof of Lemma \ref{6.18lemma2}]
		Define
	\begin{equation}
			w = w(x, t) := (1 - t)u(x) + t v(x).
	\end{equation}
By direct computation, we get
	
	\begin{align*}
		&\frac{\partial}{\partial t}(w^{\alpha}|\nabla w|^{2l}S_{k - l}) \\
		=&\alpha w^{\alpha - 1}|\nabla w|^{2l}S_{k - l}(v - u) + 2l w^{\alpha}|\nabla w|^{2l - 2}S_{k - l}\nabla w\cdot \nabla (v - u) \\
		&+ w^{\alpha}|\nabla w|^{2l}S_{k - l}^{i j}\Big[(v - u) w_{i j} + w (v - u)_{i j} - \nabla w\cdot \nabla (v - u) \delta_{i j}\Big]\\
		=&\alpha w^{\alpha - 1}|\nabla w|^{2l}S_{k - l}(v - u) + 2l w^{\alpha}|\nabla w|^{2l - 2}S_{k - l}\nabla w\cdot \nabla (v - u) \\
		&+ w^{\alpha - 1}|\nabla w|^{2l}S_{k - l}^{i j}(v - u) \left( A_{i j} + \frac{1}{2}|\nabla w|^2 \delta_{i j}\right) + w^{\alpha + 1}|\nabla w|^{2l}S_{k - l}^{i j}  (v - u)_{i j} \\
		&- (n + l + 1 - k)  w^{\alpha}|\nabla w|^{2l}S_{k - l - 1}\nabla w\cdot \nabla (v - u)\\
		=&(\alpha + k - l) w^{\alpha - 1}|\nabla w|^{2l}S_{k - l}(v - u) + 2l w^{\alpha}|\nabla w|^{2l - 2}S_{k - l}\nabla w\cdot \nabla (v - u) \\
		& + \frac{1}{2}(n + l + 1 - k) w^{\alpha - 1}|\nabla w|^{2l + 2}S_{k - l - 1}(v - u)   \\
		&+ w^{\alpha + 1}|\nabla w|^{2l}S_{k - l}^{i j}  (v - u)_{i j} - (n + l + 1 - k)  w^{\alpha}|\nabla w|^{2l}S_{k - l - 1}\nabla w\cdot \nabla (v - u),\\
	\end{align*}
	\begin{equation}
		\begin{aligned}
		\Rightarrow	&\frac{\partial}{\partial t}(w^{\alpha}|\nabla w|^{2l}S_{k - l}) \\
			=&(\alpha + k - l) w^{\alpha - 1}|\nabla w|^{2l}S_{k - l}(v - u)  + \frac{1}{2}(n + l + 1 - k) w^{\alpha - 1}|\nabla w|^{2l + 2}S_{k - l - 1}(v - u) \\
			& + 2l w^{\alpha}|\nabla w|^{2l - 2}S_{k - l}\nabla w\cdot \nabla (v - u) - (n + l + 1 - k)  w^{\alpha}|\nabla w|^{2l}S_{k - l - 1}\nabla w\cdot \nabla (v - u) \\
			&+ w^{\alpha + 1}|\nabla w|^{2l}S_{k - l}^{i j}  (v - u)_{i j}. \\
		\end{aligned}
	\end{equation}
So we get
	\begin{align*}
		&w^{\alpha + 1}|\nabla w|^{2l}S_{k - l}^{i j}  (v - u)_{i j} \\
		=&\Big[ w^{\alpha + 1}|\nabla w|^{2l}S_{k - l}^{i j}  (v - u)_{i}\Big]_{j} -  (w^{\alpha + 1}|\nabla w|^{2l}S_{k - l}^{i j} )_{j} (v - u)_{i} \\
		=&\Big[ w^{\alpha + 1}|\nabla w|^{2l}S_{k - l}^{i j}  (v - u)_{i}\Big]_{j} - (\alpha + 1) w^{\alpha}|\nabla w|^{2l}S_{k - l}^{i j}w_{j} (v - u)_{i}\\
		& - 2l w^{\alpha + 1}|\nabla w|^{2l - 2}S_{k - l}^{i j} w_{j m}w_{m} (v - u)_{i} \\
		&-  w^{\alpha + 1}|\nabla w|^{2l} (v - u)_{i}\cdot \Big[ -(n - k + l + 1) w^{-1} S_{k - l - 1}w_{i} + n w^{-1} S^{i j}_{k - l}w_{j} \Big]\\
		=&\Big[ w^{\alpha + 1}|\nabla w|^{2l}S_{k - l}^{i j}  (v - u)_{i}\Big]_{j} - (\alpha + n + 1) w^{\alpha}|\nabla w|^{2l}S_{k - l}^{i j}w_{j} (v - u)_{i}\\
		& -  2l w^{\alpha + 1}|\nabla w|^{2l - 2}S_{k - l}^{i j} w_{j m}w_{m} (v - u)_{i} \\
		&+  (n - k + l + 1) w^{\alpha}|\nabla w|^{2l} S_{k - l - 1} \nabla w\cdot \nabla (v - u) .
	\end{align*}
	Note that
	\begin{equation}
		wS^{ij}_{k - l}w_{j m} = S^{ij}_{k - l}\left( A_{j m} + \frac{1}{2}|\nabla w|^2 \delta_{j m} \right) = S_{k - l} \delta_{i m} - S^{im}_{k - l + 1} + \frac{1}{2}|\nabla w|^2 S^{i m}_{k - l},
	\end{equation}
	then we obtain
	\begin{align*}
		& w^{\alpha + 1}|\nabla w|^{2l}S_{k - l}^{i j}  (v - u)_{i j} \\
		=&\Big[ w^{\alpha + 1}|\nabla w|^{2l}S_{k - l}^{i j}  (v - u)_{i}\Big]_{j} -  (\alpha + n + 1) w^{\alpha}|\nabla w|^{2l}S_{k - l}^{i j}w_{j} (v - u)_{i} \\
		& -  2l w^{\alpha}|\nabla w|^{2l - 2} w_{j} (v - u)_{i}\left( S_{k - l} \delta_{i j} - S^{ij}_{k - l + 1} + \frac{1}{2}|\nabla w|^2 S^{i j}_{k - l} \right) \\
		&+  (n - k + l + 1) w^{\alpha}|\nabla w|^{2l} S_{k - l - 1} \nabla w\cdot \nabla (v - u)\\
		=&\Big[ w^{\alpha + 1}|\nabla w|^{2l}S_{k - l}^{i j}  (v - u)_{i}\Big]_{j}\\
		&+  2l w^{\alpha}|\nabla w|^{2l - 2}  S^{ij}_{k - l + 1}  w_{j} (v - u)_{i}  - (\alpha + n + l + 1) w^{\alpha}|\nabla w|^{2l}S_{k - l}^{i j}w_{j} (v - u)_{i} \\
		&-  2l w^{\alpha}|\nabla w|^{2l - 2}  S_{k - l} \nabla w\cdot \nabla (v - u) +  (n - k + l + 1) w^{\alpha}|\nabla w|^{2l} S_{k - l - 1} \nabla w\cdot \nabla (v - u).\\
	\end{align*}
	Finally we get that
	\begin{align*}
		&\frac{\partial}{\partial t}(w^{\alpha}|\nabla w|^{2l}S_{k - l}) \\
		=&(\alpha + k - l) w^{\alpha - 1}|\nabla w|^{2l}S_{k - l}(v - u)  + \frac{1}{2}(n + l + 1 - k) w^{\alpha - 1}|\nabla w|^{2l + 2}S_{k - l - 1}(v - u) \\
		& + 2l w^{\alpha}|\nabla w|^{2l - 2}S_{k - l}\nabla w\cdot \nabla (v - u) - (n + l + 1 - k)  w^{\alpha}|\nabla w|^{2l}S_{k - l - 1}\nabla w\cdot \nabla (v - u) \\
		&+ \Big[w^{\alpha + 1}|\nabla w|^{2l}S_{k - l}^{i j}  (v - u)_{i}\Big]_{j}\\
		&+  2l w^{\alpha}|\nabla w|^{2l - 2}  S^{ij}_{k - l + 1}  w_{j} (v - u)_{i}  - (\alpha + n + l + 1) w^{\alpha}|\nabla w|^{2l}S_{k - l}^{i j}w_{j} (v - u)_{i} \\
		&-  2l w^{\alpha}|\nabla w|^{2l - 2}  S_{k - l} \nabla w\cdot \nabla (v - u) +  (n - k + l + 1) w^{\alpha}|\nabla w|^{2l} S_{k - l - 1} \nabla w\cdot \nabla (v - u)\\
		=&\Big[w^{\alpha + 1}|\nabla w|^{2l}S_{k - l}^{i j}  (v - u)_{i}\Big]_{j} \\
		&+(\alpha + k - l) w^{\alpha - 1}|\nabla w|^{2l}S_{k - l}(v - u)  + \frac{1}{2}(n + l + 1 - k) w^{\alpha - 1}|\nabla w|^{2l + 2}S_{k - l - 1}(v - u) \\
		&+  2l w^{\alpha}|\nabla w|^{2l - 2}  S^{ij}_{k - l + 1}  w_{j} (v - u)_{i}  - (\alpha + n + l + 1) w^{\alpha}|\nabla w|^{2l}S_{k - l}^{i j}w_{j} (v - u)_{i}. \\
	\end{align*}
	Recalling \eqref{4.28.828}, we have
	\begin{equation}
		\begin{aligned}
			&(k + l)w^{\alpha - 1}|\nabla w|^{2l}S_{k - l} \\
			=& (w^{\alpha}|\nabla w|^{2l}S^{i j}_{k - l} w_{i})_{j} - (\alpha + n + l) w^{\alpha - 1}|\nabla w|^{2l}S^{i j}_{k - l} w_{i}w_{j} + 2l w^{\alpha - 1}|\nabla w|^{2l - 2}  S^{i j}_{k - l + 1} w_{i} w_{j} \\
			&  + \frac{1}{2}(n + l + 1 - k) w^{\alpha - 1}|\nabla w|^{2l + 2} S_{k - l - 1},\\
			\Rightarrow & w^{\alpha}|\nabla w|^{2l}S^{i j}_{k - l} w_{i} (v - u)_{j} \\
			=& \Big[w^{\alpha}|\nabla w|^{2l}S^{i j}_{k - l} w_{i}(v - u) \Big]_{j} - (k + l)w^{\alpha - 1}|\nabla w|^{2l}S_{k - l}(v - u)  \\
			&-(\alpha + n + l) w^{\alpha - 1}|\nabla w|^{2l}S^{i j}_{k - l} w_{i}w_{j}(v - u) \\
			&+ 2l w^{\alpha - 1}|\nabla w|^{2l - 2}  S^{i j}_{k - l + 1} w_{i} w_{j}(v - u)   + \frac{1}{2}(n + l + 1 - k) w^{\alpha - 1}|\nabla w|^{2l + 2} S_{k - l - 1}(v - u).\\
		\end{aligned}
	\end{equation}
	Thus we get
	\begin{align*}
		&\frac{\partial}{\partial t}(w^{\alpha}|\nabla w|^{2l}S_{k - l}) \\
		=&\Big[w^{\alpha + 1}|\nabla w|^{2l}S_{k - l}^{i j}  (v - u)_{i}\Big]_{j} \\
		&+(\alpha + k - l) w^{\alpha - 1}|\nabla w|^{2l}S_{k - l}(v - u)  + \frac{1}{2}(n + l + 1 - k) w^{\alpha - 1}|\nabla w|^{2l + 2}S_{k - l - 1}(v - u) \\
		&+  2l \Bigg\{ \Big[w^{\alpha}|\nabla w|^{2l - 2}S^{i j}_{k - l + 1} w_{i}(v - u) \Big]_{j} - (k + l - 1)w^{\alpha - 1}|\nabla w|^{2l - 2}S_{k - l + 1}(v - u)  \\
		&-(\alpha + n + l - 1) w^{\alpha - 1}|\nabla w|^{2l - 2}S^{i j}_{k - l + 1} w_{i}w_{j}(v - u) \\
		&+ (2l - 2) w^{\alpha - 1}|\nabla w|^{2l - 4}  S^{i j}_{k - l + 2} w_{i} w_{j}(v - u)   + \frac{1}{2}(n + l - k) w^{\alpha - 1}|\nabla w|^{2l } S_{k - l}(v - u)\Bigg\}\\
		&  - (\alpha + n + l + 1) \Bigg\{\Big[w^{\alpha}|\nabla w|^{2l}S^{i j}_{k - l} w_{i}(v - u) \Big]_{j} - (k + l)w^{\alpha - 1}|\nabla w|^{2l}S_{k - l}(v - u)  \\
			&-(\alpha + n + l) w^{\alpha - 1}|\nabla w|^{2l}S^{i j}_{k - l} w_{i}w_{j}(v - u) \\
			&+ 2l w^{\alpha - 1}|\nabla w|^{2l - 2}  S^{i j}_{k - l + 1} w_{i} w_{j}(v - u)   + \frac{1}{2}(n + l + 1 - k) w^{\alpha - 1}|\nabla w|^{2l + 2} S_{k - l - 1}(v - u) \Bigg\}.
	\end{align*}
	
It follows that
	\begin{align*}
		&\Rightarrow \frac{\partial}{\partial t}(w^{\alpha}|\nabla w|^{2l}S_{k - l}) \\
		=&\Big[w^{\alpha + 1}|\nabla w|^{2l}S_{k - l}^{i j}  (v - u)_{i}\Big]_{j}  +  2l  \Big[w^{\alpha}|\nabla w|^{2l - 2}S^{i j}_{k - l + 1} w_{i}(v - u) \Big]_{j}\\
		& - (\alpha + n + l + 1) \Big[w^{\alpha}|\nabla w|^{2l}S^{i j}_{k - l} w_{i}(v - u) \Big]_{j}\\
		&+\Big[ (\alpha + k - l) + l(n + l - k) + (\alpha + n + l + 1)(k + l)\Big] w^{\alpha - 1}|\nabla w|^{2l}S_{k - l}(v - u) \\
		&- 2l(k + l - 1)w^{\alpha - 1}|\nabla w|^{2l - 2}S_{k - l + 1}(v - u) \\
		&-\frac{1}{2} (\alpha + n + l ) (n + l + 1 - k) w^{\alpha - 1}|\nabla w|^{2l + 2} S_{k - l - 1}(v - u)\\
		& -4l(\alpha + n + l) w^{\alpha - 1}|\nabla w|^{2l - 2}S^{i j}_{k - l + 1} w_{i}w_{j}(v - u) \\
		&+ 2l(2l - 2) w^{\alpha - 1}|\nabla w|^{2l - 4}  S^{i j}_{k - l + 2} w_{i} w_{j}(v - u)   \\
		& + (\alpha + n + l + 1)(\alpha + n + l) w^{\alpha - 1}|\nabla w|^{2l}S^{i j}_{k - l} w_{i}w_{j}(v - u) .
	\end{align*}

	So
	\begin{align*}
		\frac{\partial}{\partial t} \psi =&\sum_{l = 0}^k b_{l} \frac{\partial}{\partial t}(w^{\alpha}|\nabla w|^{2l}S_{k - l}) \\
		=&\sum_{l = 0}^k  b_{l} \Bigg\{ \Big[w^{\alpha + 1}|\nabla w|^{2l}S_{k - l}^{i j}  (v - u)_{i}\Big]_{j} +  2l  \Big[w^{\alpha}|\nabla w|^{2l - 2}S^{i j}_{k - l + 1} w_{i}(v - u) \Big]_{j}\\
		& - (\alpha + n + l + 1) \Big[w^{\alpha}|\nabla w|^{2l}S^{i j}_{k - l} w_{i}(v - u) \Big]_{j}\\
		&+ \Big[ \alpha + k - l + l(n + l - k) + (\alpha + n + l + 1)(k + l)\Big] w^{\alpha - 1}|\nabla w|^{2l}S_{k - l}(v - u) \\
		&- 2l(k + l - 1)w^{\alpha - 1}|\nabla w|^{2l - 2}S_{k - l + 1}(v - u) \\
		&-\frac{1}{2} (\alpha + n + l ) (n + l + 1 - k) w^{\alpha - 1}|\nabla w|^{2l + 2} S_{k - l - 1}(v - u)\\
		& -4l(\alpha + n + l) w^{\alpha - 1}|\nabla w|^{2l - 2}S^{i j}_{k - l + 1} w_{i}w_{j}(v - u) \\
		&+ 2l(2l - 2) w^{\alpha - 1}|\nabla w|^{2l - 4}  S^{i j}_{k - l + 2} w_{i} w_{j}(v - u)   \\
		& + (\alpha + n + l + 1)(\alpha + n + l) w^{\alpha - 1}|\nabla w|^{2l}S^{i j}_{k - l} w_{i}w_{j}(v - u)\Bigg\}.\\
	\end{align*}
	At last we obtain
\begin{align}\label{7.22.19}
		\frac{\partial}{\partial t} \psi =&\sum_{l = 0}^k  b_{l}  \Big[w^{\alpha + 1}|\nabla w|^{2l}S_{k - l}^{i j}  (v - u)_{i}\Big]_{j} +  2 \sum_{l = 0}^k  lb_l  \Big[w^{\alpha}|\nabla w|^{2l - 2}S^{i j}_{k - l + 1} w_{i}(v - u) \Big]_{j}\notag\\
	& - \sum_{l = 0}^k (\alpha + n + l + 1)b_l \Big[w^{\alpha}|\nabla w|^{2l}S^{i j}_{k - l} w_{i}(v - u) \Big]_{j}\notag\\
	&+  \Big\{ b_{0} \Big[ \alpha + k  + (\alpha + n + 1)k\Big] - 2k b_{1} \Big\} w^{\alpha - 1}S_{k}(v - u)\notag \\
	&+ \Bigg\{ b_{k}\Big[ \alpha + kn + 2k(\alpha + n + k + 1)\Big] -\frac{1}{2} (\alpha + n + k - 1 ) n b_{k - 1} -4k(\alpha + n + k)b_{k} \notag\\
	&+ b_{k - 1}(\alpha + n + k)(\alpha + n + k - 1) \Bigg\} w^{\alpha - 1}|\nabla w|^{2k}(v - u) \notag\\
	&+ \sum_{l = 1}^{k - 1} \Bigg\{ b_{l} \Big[ \alpha + k - l + l(n + l - k) + (\alpha + n + l + 1)(k + l)\Big] - 2(l + 1)(k + l)b_{l + 1} \notag\\
	& -\frac{1}{2} (\alpha + n + l - 1) (n + l - k)b_{l - 1}\Bigg\} w^{\alpha - 1}|\nabla w|^{2l}S_{k - l}(v - u) \notag\\
	&+ \sum_{l = 0}^{k - 2}\Bigg[ b_{l} (\alpha + n + l + 1)(\alpha + n + l) -4(l + 1)(\alpha + n + l + 1)b_{l + 1}\notag\\
	&+ 4(l + 2)(l + 1)b_{l + 2} \Bigg]  w^{\alpha - 1}|\nabla w|^{2l}S^{i j}_{k - l} w_{i}w_{j}(v - u).
\end{align}
	Recall that 
	\begin{equation}
		a_{l} = - \frac{(\alpha + n + l)!}{2^{l}\cdot l!(\alpha + n)!}, \text{ for $l = 0, 1, \cdots, k - 1$,}
	\end{equation}
	and
	\begin{equation}\label{6.19b}
		\begin{aligned}
			b_{l + 1} &=  \Big[(k + l + 1)\alpha + kn + 2k(l + 1) \Big]\cdot \frac{(\alpha + n + l)!}{2^{l + 1}(l + 1)! (\alpha + n)!}, 
		\end{aligned}
	\end{equation}
	with $0 \leq l \leq k - 1$ and $b_{0} = k.$
	Therefore we get 
	\begin{equation}
		\begin{aligned}
			\frac{\partial}{\partial t} \psi &= \sum_{l = 0}^k  b_{l}  \Big[w^{\alpha + 1}|\nabla w|^{2l}S_{k - l}^{i j}  (v - u)_{i}\Big]_{j} +  2\sum_{l = 0}^{k} lb_l  \Big[w^{\alpha}|\nabla w|^{2l - 2}S^{i j}_{k - l + 1} w_{i}(v - u) \Big]_{j}\\
		& -\sum_{l = 0}^{k} (\alpha + n + l + 1) b_l\Big[w^{\alpha}|\nabla w|^{2l}S^{i j}_{k - l} w_{i}(v - u) \Big]_{j}.\\
		\end{aligned}
	\end{equation}
	\end{proof}
	
	Now we can get the following monotonicity formula:

	\begin{proof}[Proof of Theorem  \ref{5.6monotonicity}]
		By Lemma \ref{6.18lemma2}, we get 
	\begin{equation}\label{7.22mon1}
		\frac{\partial}{\partial t}\int_{\Omega} \psi = \int_{\partial\Omega}\sum_{l = 0}^k  b_{l}  w^{\alpha + 1}|\nabla w|^{2l}S_{k - l}^{i j}  (v - u)_{i}\nu_{j} \geq 0,
	\end{equation}
	where $\nu(x) = (\nu_1, \cdots, \nu_n)$ is the outer normal vector of $\partial \Omega$ at $x \in \partial \Omega$.
	\end{proof}

	Using monotonicity formula and \eqref{7.22.19} , we can get one comparison principle directly. The proof is different from the Hessian measure case in \cite{MR1634570} and \cite{MR1726702}.
	\begin{theorem}\label{7.22comparison}
	Let $b_0 = k$ and $b_{l} \geq 0$ for $1 \leq l \leq k$. Suppose $\{ b_i\}_{i = 0}^k$ satisfies:
	\begin{numcases}{}
		\label{7.22.20.14} b_{0} \Big[ \alpha + k  + (\alpha + n + 1)k\Big] - 2k b_{1} < 0,\\
		 b_{k}\Big[ \alpha + kn + 2k(\alpha + n + k + 1)\Big] -\frac{1}{2} (\alpha + n + k - 1 ) n b_{k - 1} -4k(\alpha + n + k)b_{k} \notag\\
	\label{7.22.20.15}	+ b_{k - 1}(\alpha + n + k)(\alpha + n + k - 1) \leq 0,\\
		b_{l} \Big[ \alpha + k - l + l(n + l - k) + (\alpha + n + l + 1)(k + l)\Big] - 2(l + 1)(k + l)b_{l + 1} \notag\\
	\label{7.22.20.16}	 -\frac{1}{2} (\alpha + n + l - 1) (n + l - k)b_{l - 1} \leq 0 \text{ \quad for $1 \leq l \leq k - 1$},\\
	 b_{l} (\alpha + n + l + 1)(\alpha + n + l) -4(l + 1)(\alpha + n + l + 1)b_{l + 1} \notag\\
	\label{7.22.20.17}	+ 4(l + 2)(l + 1)b_{l + 2} \leq 0 \text{ \quad for $0 \leq l \leq k - 2$}.
	\end{numcases}

		If $u, v \in C^{2}(\bar{\Omega}) \cap \Phi^k(\Omega)$ satisfy
		\begin{numcases}{}
			\label{7.22mon2} \sum_{l = 0}^k b_{l} v^{\alpha}|\nabla v|^{2l}S_{k - l}(v) \leq  \sum_{l = 0}^k b_{l} u^{\alpha}|\nabla u|^{2l}S_{k - l}(u) \text{ in $\Omega$},\\
			v\geq u \text{ on $\partial\Omega$}.	
		\end{numcases}
		Then $u \leq v$ in $\Omega$.
	\end{theorem}
	\begin{proof}[Proof of Theorem \ref{7.22comparison}]
		Define $\Omega_1 := \{ x\in \Omega : u(x) > v(x)\}$ and $\psi(u) = \sum_{l = 0}^k b_{l} u^{\alpha}|\nabla u|^{2l}S_{k - l}(u) $. Suppose the set $\Omega_1$ is not empty. Then we know that
		\begin{numcases}{}
			u > v \text{ in $\Omega_1$},\\
			v = u \text{ on $\partial\Omega_1$}.	
		\end{numcases}
		Let $w: = (1 - t)u + t v$.
		Then combining \eqref{7.22.19} and \eqref{7.22mon2} and using conditions \eqref{7.22.20.15}, \eqref{7.22.20.16}, \eqref{7.22.20.17}, we get 
	\begin{align*}
		0 &\geq \int_{\Omega_1}\psi(v) - \int_{\Omega_1}\psi(u)\\
		&\geq \int_0^1 \frac{\partial}{\partial t}\int_{\Omega_1} \psi(w) dt\\
		&\geq \int_0^1\sum_{l = 0}^k \int_{\Omega_1} b_{l}  \Big[w^{\alpha + 1}|\nabla w|^{2l}S_{k - l}^{i j}  (v - u)_{i}\Big]_{j} +  2 \sum_{l = 0}^k  lb_l  \Big[w^{\alpha}|\nabla w|^{2l - 2}S^{i j}_{k - l + 1} w_{i}(v - u) \Big]_{j}\notag\\
		& - \sum_{l = 0}^k (\alpha + n + l + 1)b_l \Big[w^{\alpha}|\nabla w|^{2l}S^{i j}_{k - l} w_{i}(v - u) \Big]_{j}\notag\\
		&+  \Big\{ b_{0} \Big[ \alpha + k  + (\alpha + n + 1)k\Big] - 2k b_{1} \Big\} w^{\alpha - 1}S_{k}(v - u) dt\\
		&\geq \int_0^1   \int_{\Omega_1}  \Big\{ b_{0} \Big[ \alpha + k  + (\alpha + n + 1)k\Big] - 2k b_{1} \Big\} w^{\alpha - 1}S_{k}(v - u) dt.\\
	\end{align*}
		By condition \eqref{7.22.20.14} , we obtain
		\begin{align}
	\int_0^1	\int_{\Omega_1} w^{\alpha - 1}S_{k}(v - u) dxdt = 0.
		\end{align}
	Since $w, u - v > 0$ in $\Omega$, we deduce that $S_{k}(w) = 0$ in $\Omega_1$. By the proof of Lemma \ref{5.7sum} and the fact that $\sigma_{k}^{\frac{1}{k}}$ is concave, we get 
	\begin{align}
		&v u^{-1}|\nabla u|^2 + uv^{-1}|\nabla v|^2 - 2\nabla u \cdot \nabla v = 0 \text{ in $\Omega_1$},\\
		\Rightarrow &|\nabla \log u - \nabla \log v| = 0 \text{ in $\Omega_1$},\\
		\Rightarrow & u \equiv v \text{ in $\Omega_1$},
	\end{align}
		which is impossible. Therefore the set $\Omega_1$ is empty.
	\end{proof}
	Such $\{ b_i\}_{i = 0}^k$ exists for some special $\alpha$. For example, let $b_0 = k$, $b_{l} = 0$ for $1 \leq l \leq k$, $k < \frac{n}{2}$ and $\alpha = - n$, then the conditions \eqref{7.22.20.14} - \eqref{7.22.20.17}  all hold.
		\begin{corollary}\label{7.22coro}
		If $n > 2k$ and $u, v \in C^{2}(\bar{\Omega}) \cap \Phi^k(\Omega)$ satisfy
		\begin{numcases}{}
			v^{-n}S_{k}(v) \leq u^{-n} S_{k}(u) \text{ in $\Omega$},\\
			v\geq u \text{ on $\partial\Omega$}.	
		\end{numcases}
		Then $u \leq v$ in $\Omega$.
	\end{corollary}

	\section{Locally uniform convergence for $n < 2k$}
	In this section, we will prove Theorem \ref{mainresult6.18} for the case $n < 2k$. 	Firstly by  Lemma \ref{5.6.sigma}, and the  Theorem 2.7 in \cite{MR1726702} and get:
	\begin{lemma}\label{6.21.2k>n}
		For $n < 2k$, we have $\Phi^k(\Omega) \subset C^{0,\alpha}(\Omega)$ for $\alpha = 2 - \frac{n}{k}$ and for any $\Omega_2\subset \subset \Omega_1\subset \subset \Omega$ and $u \in \Phi^k(\Omega)$,
		\begin{equation}
			\| u\|_{C^{0,\alpha}(\Omega_2)}\leq C \int _{\Omega_1} u.
		\end{equation}
	\end{lemma}
	So in this case, $u^{(m)}$ converges to $u$ in $L^1_{loc}(\Omega)$ means that $u^{(m)}$ converges to $u$ in $C^0_{loc}(\Omega)$. Then we have the following result:
	
	\begin{theorem}\label{mainresult}
		For any $u \in \Phi^k(\Omega)$, there exists a Borel measure $\mu_{k}[u]$ in $\Omega$ such that
		\begin{itemize}
			\item $\mu_{k}[u] = S_{k} (u)$ for $u \in C^2(\Omega)$.
			\item If $\{ u^{(m)} \}$ is a sequence in $\Phi^k(\Omega)$ converging to $u$ locally uniformly in $\Omega$, then the corresponding measure $\mu_{k}[u^{(m)}] \rightarrow \mu_{k}[u]$ weakly.
		\end{itemize}
		
	\end{theorem}

	Recalling the identity \eqref{5.6.equ1} and Lemma \ref{5.6.sigma}, together with the Lemma 2.2 in \cite{MR1634570} and Theorem 3.1 in \cite{MR1726702}, we get:
	\begin{lemma}\label{5.6bound}
		Let $u \in \Phi^k(\Omega) \cap C^2(\Omega)$ satisfy that $ |u| \leq M$ in $\Omega_1 \subset \subset \Omega.$ Then 
		\begin{equation}
			\int_{\Omega_2} \sum_{l = 0}^k |\nabla u|^{2l} S_{k - l}  \leq C (\text{osc}_{\Omega_1}\ u)^{2k} , \text{if $\Omega_2 \subset\subset \Omega_1$.}
		\end{equation}
	\end{lemma}
%
	Next, let us introduce a special convex function in $B_{2} := \{ x \in \mathbb R^n: |x| \leq 2 \}$.	
	Define 
	\begin{equation}
		\begin{aligned}
			\eta &= C, \text{ if $|x| \leq 1$,}\\
			\eta &= (r - 1)^3 + C, \text{ if $1 \leq |x| \leq 2$,}\\
		\end{aligned}
	\end{equation}
	where $C > 0$ is a constant, we know  $\eta \in C^2(\mathbb R^n)$.
	
	\begin{lemma}\label{5.2cutoff}
		There exists $C_{0}$ depending only on $n$, such that the matrix 
		\begin{equation}
			A_{i j}(\eta) = \eta \eta_{i j} - \frac{1}{2} |\nabla \eta|^2 \delta_{i j},
		\end{equation}
		is nonnegative-definite in $B_{2}$ for any $C > C_{0}$. 
	\end{lemma}
	\begin{proof}[Proof of Lemma \ref{5.2cutoff}]
		Recall that
		\begin{align*}
		r_{i} &= \frac{x_{i}}{r},\\
		r_{i j}&= \frac{1}{r}\delta_{i j} - \frac{x_i x_j}{r^3},
	\end{align*}
	with $r = |x|$. Then we know that when $r > 1$
	\begin{align*}
		\eta_{i} &= 3(r - 1)^2 \frac{x_{i}}{r},\\
		\eta_{i j} &= 6(r - 1)\frac{x_{i} x_{j}}{r^2} + 3(r - 1)^2\left( \frac{1}{r}\delta_{i j} - \frac{x_i x_j}{r^3} \right)\\
		&= \frac{3(r - 1)(r + 1)}{r^3}x_{i} x_{j} + 3(r - 1)^2\frac{1}{r}\delta_{i j},
	\end{align*}
	and
	\begin{align*}
		A_{i j}(\eta) &= \Big[ (r - 1)^3 + C \Big] \cdot \Big[ \frac{3(r - 1)(r + 1)}{r^3}x_{i} x_{j} + 3(r - 1)^2\frac{1}{r}\delta_{i j} \Big] - \frac{9}{2}(r - 1)^4\delta_{i j}\\
		&= \Big[ (r - 1)^3 + C \Big] \cdot\frac{3(r - 1)(r + 1)}{r^3}x_{i} x_{j} + \Big[ \frac{(r - 1)^3 + C}{r}   - \frac{3}{2}(r - 1)^2 \Big] 3(r - 1)^2\delta_{i j}. 
	\end{align*}
	So if we let $C > 3$, then the matrix $(A_{i j}(\eta))$ is always nonnegative-definite in $B_{2}$. 
	\end{proof}
	Now we can give the proof for $n < 2k$:
	
	\begin{proof}[Proof of Theorem \ref{mainresult}]
	We divide this proof into two steps.

	\begin{itemize}
		\item Step1:
	Firstly we can follow the idea in \cite{MR1634570} and \cite{zhangwei} to prove the weak convergence of $\psi_k^\alpha$, which is defined by \eqref{5.6psi}.
		Suppose $u \in \Phi^k(\Omega), \{ u^{(m)} \} \subset \Phi^k(\Omega) \cap C^2 (\Omega)$ and $u^{(m)} \rightarrow u$ locally uniformly in $\Omega$. By Lemma \ref{5.6bound}, the integrals
		\begin{equation}
			\int_{\Omega'}\psi_{k}^{\alpha}(u^{(m)}) 
		\end{equation}
	\noindent are uniformly bounded for any subdomain $\Omega'\subset\subset \Omega$(the bound also depends on $\alpha$). Hence there is a subsequnce $\{ \psi_{k}^{\alpha}(u^{(m_p)})\}$ that converges weakly to a Borel measure $\mu_{k}^{\alpha}[u]$. Firstly we need to prove that the measure $\mu_{k}^{\alpha}[u]$ is uniquely determined by the function $u$. Assume there exist two sequences $\{u^{(m)}\},$ $\{ v^{(m)} \} \subset \Phi^k(\Omega) \cap C^2 (\Omega)$ which both converge to $u$ locally uniformly, but the corresponding sequences $\{ \psi_{k}^{\alpha}(u^{(m)}) \}$ and $\{ \psi_{k}^{\alpha}(v^{(m)}) \}$ weakly converge to Borel measure $\tilde \mu_1$ and $\tilde \mu_2$, respectively. Let $B_{R} = B_{R}(x_0) \Subset \Omega$ and fix some $\sigma \in (0, 1)$. Define 
		\begin{equation}
		\begin{aligned}
			\eta &= \tilde C, \text{ if $|x - x_0| \leq \sigma R$,}\\
			\eta &= \frac{(r - \sigma R)^3}{(R - \sigma R)^3} + \tilde C, \text{ if $\sigma R \leq |x| \leq R$.}\\
		\end{aligned}
	\end{equation}
	Then by Lemma \ref{5.2cutoff}, we know that when $\tilde C$ is large enough, then the matrix 
		\begin{equation}
			A_{i j}(\eta) = \eta \eta_{i j} - \frac{1}{2} |\nabla \eta|^2 \delta_{i j},
		\end{equation}
		is nonnegative-definite in $B_{R}$. So $\eta \in \Phi^k(B_R)$. For fixed $\varepsilon > 0$, it then follows from the uniform convergence of $\{ u^{(m)}\}$ and $\{ v^{(m)} \}$ on $\bar B_{R}$, that
		\begin{equation}
			-\frac{\varepsilon}{2} \leq u^{(m)} - v^{(m)} \leq \frac{\varepsilon}{2} \text{ in $\bar B_{R}$},
		\end{equation}
		for sufficiently large $m$. Hence 
		\begin{equation}
			u^{(m)} + \varepsilon \left( \frac{1}{2} + \tilde C \right) \leq v^{(m)} + \varepsilon \eta \text{ on $\partial\bar B_{R}$.}
		\end{equation}
		Define 
		\begin{equation}
			\Omega_{m} := \left\{ x\in B_{R}: u^{(m)} + \varepsilon \left( \frac{1}{2} + \tilde C \right) > v^{(m)} + \varepsilon \eta \right\}.
		\end{equation}
		Without loss of generality, we may assume that $\partial\Omega_{m}$ is sufficiently smooth so that from Lemma \ref{5.6monotonicity}, when $\alpha \geq 0$,
		\begin{equation}
			\int_{\Omega} \psi_{k}^{\alpha}\left(u^{(m)} + \varepsilon \left( \frac{1}{2} + \tilde C \right)\right) dV \leq \int_{\Omega} \psi_{k}^{\alpha}(v^{(m)} + \varepsilon \eta) dV.
		\end{equation}
		Using $\alpha \geq 0$ and Lemma \ref{5.6.sigma}, and expanding $S_{k}\left(u^m + \varepsilon \left( \frac{1}{2} + \tilde C \right)\right)$ as the sum of mixed k-Hessian operators, we get that 
		\begin{equation}
			\int_{\Omega} \psi_{k}^{\alpha}(u^{(m)}) dV \leq \int_{\Omega} \psi_{k}^{\alpha}\left(u^{(m)} + \varepsilon \left( \frac{1}{2} + \tilde C \right)\right) dV.
		\end{equation}
		Recalling Lemma \ref{5.6bound} and expanding $\psi_{k}^{\alpha}(v^{(m)} + \varepsilon \eta)$ as the sum of mixed k-Hessian operators,
		\begin{equation}
			\int_{\Omega} \psi_{k}^{\alpha}(v^{(m)} + \varepsilon \eta) dV \leq \int_{\Omega} \psi_{k}^{\alpha}(v^{(m)}) dV + \varepsilon C,
		\end{equation}
		where the constant $C$ depends on $n, k, \sigma, u, R, \alpha$.
		Since $\eta = \tilde C$ in $B_{\sigma R}$, by the definition of $\Omega_m$, we have $B_{\sigma R} \subset \Omega_m$ and hence
		\begin{equation}
			\int_{B_{\sigma R} } \psi_{k}^{\alpha}(u^{(m)}) dV \leq \int_{B_{\sigma R} } \psi_{k}^{\alpha}(v^{(m)}) dV + \varepsilon C
		\end{equation}
		Letting $\varepsilon \rightarrow 0$ and $m \rightarrow \infty$, we then obtain
		\begin{equation}
			\tilde \mu_1(B_{\sigma R} )\leq \tilde \mu_2(B_{\sigma R} ).
		\end{equation}
		By interchanging $\{ u^{(m)}\}$ and $\{ v^{(m)} \}$, we have $\tilde \mu_1(B_{\sigma R} ) = \tilde \mu_2(B_{\sigma R} )$.

		\item Step2: By the conclusions of Step 1 , we get the weak convergence of the Borel measure 
		 
		 \begin{equation}
		 	\psi_{k}^{\alpha}(u)  = \sum_{l = 0}^{k} b_{l}(\alpha) u^{\alpha} |\nabla u|^{2l}S_{k - l}(u).
		 \end{equation}
		By the locally uniform convergence of $u^{(m)}$, we 
		get the weak convergence of the Borel measure  
		 \begin{equation}
		 	\varphi_{k}^{\alpha}(u) : = \sum_{l = 0}^{k} b_{l}(\alpha)  |\nabla u|^{2l}S_{k - l}(u).
		 \end{equation}
		Recall \eqref{6.19b}:
		\begin{equation}
			\begin{aligned}
				b_{l + 1} &=  \Big[(k + l + 1)\alpha + kn + 2k(l + 1) \Big]\cdot \frac{(\alpha + n + l)!}{2^{l + 1}(l + 1)! (\alpha + n)!},
			\end{aligned}
		\end{equation}
		with $0 \leq l \leq k - 1$ and $b_{0} = k.$ We find that for any $0 \leq l \leq k$, $b_{l} \approx C_{l} \alpha^{l}$ when $\alpha > 0$ is large enough. So there always exist $\alpha_0, \alpha_1, \cdots , \alpha_{k} > 0$, such that 
		\begin{equation}
			\det (H_{i j}) \neq 0,
		\end{equation}
		where 
		\begin{equation}
			H_{i j} = b_j(\alpha_i), \text{ with $i, j = 0, \cdots, k$.}
		\end{equation}
		So for any $0 \leq l \leq k$, we get 
		the weak convergence of the Borel measure  
		 \begin{equation}
		 	   |\nabla u|^{2l}S_{k - l}(u).
		 \end{equation}
	\end{itemize}

	\end{proof}
	By the above proof, in fact we have already got:	
	\begin{corollary}\label{coro1}
		For any $u \in \Phi^k(\Omega)$, $\alpha \in \mathbb{R}$ , $0 \leq l \leq k$ and $n < 2k$, there exists a Borel measure $\mu_{k, l, \alpha}[u]$ in $\Omega$ such that 
		\begin{itemize}
			\item $\mu_{k, l, \alpha}[u] = u^{\alpha} |\nabla u|^{2l}S_{k - l} (u)$ for $u \in C^2(\Omega)$.
			\item If $\{ u^{(m)} \}$ is a sequence in $\Phi^k(\Omega)$ converging locally in measure to a function $u \in \Phi^k(\Omega)$, then the corresponding measure $\mu_{k, l, \alpha}[u^{(m)}] \rightarrow \mu_{k, l, \alpha}[u]$ weakly.
		\end{itemize}
	\end{corollary}
	Define 
	\begin{equation}
		\begin{aligned}
			\Psi^k(\Omega) &:= \Big\{ u \in C^0(\bar\Omega): u > 0, \text{and there exists a sequence $ \{ u^{(m)} \} \in \tilde\Phi^k(\Omega)$ },\\
			& \text{ such that $u^{(m)} $ converges to $u$ uniformly in $\bar \Omega$ }\Big\}.
		\end{aligned}
	\end{equation}
	Then using Corollary \ref{coro1}, we can extend the result in Corollary \ref{7.22coro} to $\Psi^k(\Omega)$.
		\begin{corollary}\label{7.22coro2}
		If $n > 2k$ and $u, v \in  \Psi^k(\Omega)$ satisfy
		\begin{numcases}{}
			v^{-n}S_{k}(v) \leq u^{-n} S_{k}(u) \text{ in $\Omega$},\\
			v\geq u \text{ on $\partial\Omega$}.	
		\end{numcases}
		Then $u \leq v$ in $\Omega$.
	\end{corollary}
	
\section{$L^{\infty}$ estimate}	
	In this section, we will give interior $L^{\infty}$ bound with respect to the $\sigma_k$-Yamabe operator and prove Lemma \ref{6.18Moser}.  We use the idea from Gonz\'alez \cite{MR2263673} to use the Moser iteration, see also Li-Nguyen-Wang	\cite{MR4691488}.

\begin{lemma}\label{6.18Moser}
	If $v \in \Phi^k(B_3)$ , for $1 \le k \le \frac{n}{2}$  there exists a positive constant $C$ which only depends on $n, k$ , such that we have
	\begin{equation}
		\sup_{B_1} v \leq C \| v\|_{L^{1}(B_2)} .
	\end{equation}
\end{lemma}

	\begin{proof}[Proof of Lemma \ref{6.18Moser}]
		By Lemma \ref{6.21.2k>n}, we only need to consider the case that $n \geq 2k$. 
		The proof is based on standard Moser iteration. The proof is similar to the ones in \cite{MR2247857} and  \cite{MR2263673} by Gonz\'{a}lez. 
		By \eqref{6.19b}, we know that when $\alpha + 2k > 0$,
		\begin{equation}
			b_{l} > 0, \text{ for any $0 \leq l \leq k$}.
		\end{equation}
		By Lemma \ref{6.19divergence}, we get
	\begin{equation}
		\begin{aligned}
			&\sum_{l = 0}^{k - 1} a_{l} (v^{\alpha + 1}|\nabla v|^{2l}S^{i j}_{k - l} v_{i})_{j} + \sum_{l = 0}^k b_{l} v^{\alpha}|\nabla v|^{2l}S_{k - l} = 0,\\
			\Rightarrow &\sum_{l = 0}^k b_{l} v^{\alpha}|\nabla v|^{2l}S_{k - l} = -\sum_{l = 0}^{k - 1} a_{l} (v^{\alpha + 1}|\nabla v|^{2l}S^{i j}_{k - l} v_{i})_{j}.
		\end{aligned}
	\end{equation}
	Suppose $\frac{1}{2} \leq r < R \leq 2$. Let cut-off function  $\eta \equiv 1$ in $B_r$ and $\eta \equiv 0$ in $\mathbb R^n \backslash B_R$. So we get when $\alpha > - 2k$ and $\delta > 4k$,	
	\begin{equation}\label{6.19_9.57equ1}
		\begin{aligned}
			\int\sum_{l = 0}^k b_{l} v^{\alpha}|\nabla v|^{2l}S_{k - l}\eta^{\delta} \leq \int \sum_{l = 0}^{k - 1} a_{l} v^{\alpha + 1}|\nabla v|^{2l}S^{i j}_{k - l} v_{i}(\eta^{\delta})_{j}.
		\end{aligned}
	\end{equation}
	Since $S^{ij}_{k - l}$ is nonnegative-definite, it follows that
	for any $0 \leq l \leq k - 1$, 
	\begin{equation}\label{6.19_8.17equ1}
		\begin{aligned}
			&\int a_{l} v^{\alpha + 1}|\nabla v|^{2l}S^{i j}_{k - l} v_{i}(\eta^{\delta})_{j}\\
		\leq &\int C|a_{l}|(R - r)^{-1}v^{\alpha + 1}|\nabla v|^{2l + 1}S_{k - l - 1}\eta^{\delta - 1} .
		\end{aligned}
	\end{equation}
	By Cauchy-Schwarz inequality, we have
	\begin{equation}\label{6.19_9.55equ2}
		\begin{aligned}
			&\int (R - r)^{-1} v^{\alpha + 1}|\nabla v|^{2l + 1}S_{k - l - 1}\eta^{\delta - 1}  \\
		\leq & \int C(R - r)^{-2} v^{\alpha + 2}|\nabla v|^{2l}S_{k - l - 1} \eta^{\delta - 2} + C^{-1}v^{\alpha}|\nabla v|^{2l + 2}S_{k - l - 1}\eta^{\delta} .
		\end{aligned}
	\end{equation}
	Because we are using the method of Moser iteration, it is necessary to explain the case when $\alpha$ tends to $+\infty$. There are two different cases depending on $\alpha$.
	\begin{itemize}
		\item When $-2k < \alpha < 2n$, substitute \eqref{6.19_8.17equ1} and \eqref{6.19_9.55equ2} into \eqref{6.19_9.57equ1}:
	\begin{equation}
		\int\sum_{l = 0}^k  v^{\alpha}|\nabla v|^{2l}S_{k - l}\eta^{\delta} \leq \sum_{l = 0}^{k - 1} C(R - r)^{-2} \int v^{\alpha + 2}|\nabla v|^{2l} S_{k - l - 1} \eta^{\delta - 2}.
	\end{equation}	 
		So by induction, we finally get 	
	\begin{equation}
		 	\int\sum_{l = 0}^k  v^{\alpha}|\nabla v|^{2l}S_{k - l}\eta^{\delta}  \leq C(R - r)^{-2k}\int v^{\alpha + 2k }  \eta^{\delta - 2k} .
	\end{equation}		
		
	\item When $\alpha \geq 2n$,  we have
	\begin{equation}
		\begin{aligned}
			|a_{l}| &=  \frac{(\alpha + n + l)!}{2^{l}\cdot l!(\alpha + n)!} \leq \frac{\alpha^{l}}{ l!} , \text{ for $l = 0, 1, \cdots, k - 1$,}
		\end{aligned}
	\end{equation}
	and
	\begin{equation}
		\begin{aligned}
			b_{l + 1} &=  \Big[(k + l + 1)\alpha + kn + 2k(l + 1) \Big]\cdot \frac{(\alpha + n + l)!}{2^{l + 1}(l + 1)! (\alpha + n)!}\\
			&\geq \frac{(k + l + 1)\alpha^{l + 1}}{2^{l + 1}(l + 1)!}, \text{ for $l = -1, 0, 1,  \cdots, k - 1$.}
		\end{aligned}
	\end{equation}
	In this case, the same as above argument , by \eqref{6.19_9.57equ1} we get 
	\begin{equation}
		\begin{aligned}
				&\int\sum_{l = 0}^k \alpha^{l} v^{\alpha}|\nabla v|^{2l}S_{k - l}\eta^{\delta} \\
			\leq& C\int \sum_{l = 0}^{k - 1} a_{l} v^{\alpha + 1}|\nabla v|^{2l}S^{i j}_{k - l} v_{i}(\eta^{\delta})_{j} \\
			\leq&\sum_{l = 0}^{k - 1} C|a_{l}|\int (R - r)^{-1}v^{\alpha + 1}|\nabla v|^{2l + 1}S_{k - l - 1}\eta^{\delta - 1}\\
			\leq &\sum_{l = 0}^{k - 1} C\alpha^{l}\int (R - r)^{-1}v^{\alpha + 1}|\nabla v|^{2l + 1}S_{k - l - 1}\eta^{\delta - 1}\\
			\leq &\sum_{l = 0}^{k - 1} C\alpha^{l - 1}\int (R - r)^{-2}v^{\alpha + 2}|\nabla v|^{2l}S_{k - l - 1}\eta^{\delta - 2} + C^{-1} \alpha^{l + 1}\int v^{\alpha }|\nabla v|^{2l + 2}S_{k - l - 1}\eta^{\delta },
		\end{aligned}
	\end{equation}
	the last inequality follows from Cauchy-Schwarz inequality.
	Then we deduce that
	\begin{equation}
		\begin{aligned}
			&\int\sum_{l = 0}^k \alpha^{l} v^{\alpha}|\nabla v|^{2l}S_{k - l}\eta^{\delta} 	\leq \sum_{l = 0}^{k - 1} C\alpha^{l - 1}\int (R - r)^{-2}v^{\alpha + 2}|\nabla v|^{2l}S_{k - l - 1}\eta^{\delta - 2}.
		\end{aligned}
	\end{equation}
	By induction, we finally get that
		\begin{equation}
		\begin{aligned}
			&\int\sum_{l = 0}^k \alpha^{l} v^{\alpha}|\nabla v|^{2l}S_{k - l}\eta^{\delta} 	\leq  C\alpha^{ - k}\int (R - r)^{-2k}v^{\alpha + 2k}\eta^{\delta - 2k},
		\end{aligned}
	\end{equation}
	
	\begin{equation}
		\Rightarrow \alpha^{2k} \int  v^{\alpha}|\nabla v|^{2k}\eta^{\delta} \leq C(R - r)^{-2k}\int v^{\alpha + 2k} \eta^{\delta - 2k}.
	\end{equation}
	Let $\delta = 4k$, then
	\begin{align*}
		&\int \left\vert\nabla (v^{\frac{\alpha + 2k}{2k}}\eta^2)\right\vert^{2k} \leq C(R - r)^{-2k}\int v^{\alpha + 2k} \eta^{2k},\\
		\Rightarrow &\| \nabla (v^{\frac{\alpha + 2k}{2k}}\eta^2)\|_{L^{2k}(B_R)} \leq C(R - r)^{-1} \|v^{\frac{\alpha + 2k}{2k}}  \|_{L^{2k}(B_R)}. 
	\end{align*}
	By Sobolev inequality, it holds that
	\begin{align*}
		 &\|v^{\frac{\alpha + 2k}{2k}} \|_{L^{\frac{2k n}{n - 2k}}(B_r)} \leq C_1  \| \nabla (v^{\frac{\alpha + 2k}{2k}}\eta^2)\|_{L^{2k}(B_R)} \leq C(R - r)^{-1} \|v^{\frac{\alpha + 2k}{2k}}  \|_{L^{2k}(B_R)},\\
		 \Rightarrow &\| v\|_{L^{\frac{(\alpha + 2k)n}{n - 2k}}(B_r)} \leq C^{\frac{2k}{\alpha + 2k}} (R - r)^{-\frac{2k}{\alpha + 2k}} \| v\|_{L^{\alpha + 2k}(B_R)}.
	\end{align*}
	Following the proof of Theorem 4.1 in \cite{MR2777537} and using the fact that
	\begin{equation}
		\| v\|_{L^{p_1}(B_1)} \leq C \| v\|_{L^{p_2}(B_1)}, \text{ if $p_2 > p_1 > 0$},
	\end{equation}
	we finally get that
	\begin{equation}
		\sup_{B_1} v \leq C \| v\|_{L^{1}(B_2)} .
	\end{equation}	
		
	\end{itemize}

	\end{proof}
	Now we can give the local uniform bound:
	\begin{lemma}\label{6.17bound}
		Let $u \in \Phi^k(B_3) \cap C^2(B_3)$ . Then for $\alpha + 2k> 0$ there exists a positive constant $C$ which only depends on $n, k, \alpha$ such that
		\begin{equation}
			\int_{B_1} \sum_{l = 0}^k u^{\alpha} |\nabla u|^{2l} S_{k - l}  \leq C \|u\|_{L^1(B_2)}^{2k + \alpha} .
		\end{equation}
	\end{lemma}
	\begin{proof}[Proof of Lemma \ref{6.17bound}]
		Let cut-off function  $\eta \equiv 1$ in $B_1$ and  $\eta \equiv 0$ in $B_3\backslash B_2$.
		By Lemma \ref{6.19divergence}
		\begin{align*}
			&\int \sum_{l = 0}^k b_{l} u^{\alpha} |\nabla u|^{2l} S_{k - l} \eta^{\delta}\\
			 =& -\int 	\sum_{l = 0}^{k - 1} a_{l} (u^{\alpha + 1}|\nabla u|^{2l}S^{i j}_{k - l} u_{i})_{j}\eta^{\delta}\\
			=& \int 	\sum_{l = 0}^{k - 1} a_{l} u^{\alpha + 1}|\nabla u|^{2l}S^{i j}_{k - l} u_{i}(\eta^{\delta})_{j}\\
			\leq &C  \int 	\sum_{l = 0}^{k - 1}u^{\alpha + 1}|\nabla u|^{2l + 1}S_{k - l - 1}\eta^{\delta - 1}.
		\end{align*}
		Using Cauchy-Schwarz inequality and Lemma \ref{5.6.sigma}, we get
		\begin{align*}
			\Rightarrow &\int \sum_{l = 0}^k b_{l} u^{\alpha} |\nabla u|^{2l} S_{k - l} \eta^{\delta} \leq  C  \int 	\sum_{l = 0}^{k - 1} u^{\alpha + 2} |\nabla u|^{2l} S_{k - l - 1} \eta^{\delta - 2}\\
		\Rightarrow &\int u^{\alpha + k}\sigma_{k}\eta^{\delta}	+ \int \sum_{l = 0}^k b_{l} u^{\alpha} |\nabla u|^{2l} S_{k - l} \eta^{\delta} \leq  C  \int  u^{\alpha + k + 1}  \sigma_{k - 1} \eta^{\delta - 2} + C\int u^{\alpha + 2k}\eta^{\delta - 2k}.
		\end{align*}

		By iteration, we get if $\delta = 4k$,
		\begin{align*}
			\Rightarrow & \int \sum_{l = 0}^k  u^{\alpha} |\nabla u|^{2l} S_{k - l} \eta^{\delta} \leq C \|u\|_{L^1(B_2)}^{2k + \alpha} .
		\end{align*}
	\end{proof}
	
	\section{Weak continuity for $n \geq 2k$.}
	In this section, we will prove 	 Theorem \ref{mainresult6.18} when $n \geq 2k$.
	\begin{proof}[Proof of Theorem \ref{mainresult6.18}]
	We divide the proof into two steps. Suppose $u \in \Phi^k(\Omega), \{ u^{(m)} \} \subset \Phi^k(\Omega) \cap C^2 (\Omega)$ and $u^{(m)} $  converge locally in measure to $u$. 
	\begin{itemize}
		\item Step1: Firstly we will show that for any $\alpha > 0$ large enough, it follows that $\psi_k^{\alpha}(u^{(m)}) \rightarrow \psi_k^{\alpha}(u)$ weakly.
		
		 In this step, we follow the idea of Trudinger-Wang in \cite{MR1726702}. 
		Define
	\begin{equation}
			w = w(x, t) := (1 - t)u^{(m_1)}(x) + t u^{(m_2)}(x),
	\end{equation}
	and let $\eta \equiv 1 $ in $B_r(y)$ and $\eta \equiv 0$ in $\Omega \backslash B_{2r}(y)$ with  $B_{3r}(y) \subset \subset \Omega$.
	Then by Lemma \ref{6.18lemma2}
		\begin{align*}
			&\int_{B_{2r(y)}} \eta \psi^{\alpha}_{k}(u^{(m_2)}) - \eta \psi^{\alpha}_{k}(u^{(m_1)})\\
			 =& \int_0^1 \int_{B_{2r(y)}} \eta \frac{\partial}{\partial t}\psi^{\alpha}_{k}(w)\\
			=&\sum_{l = 0}^k \int_0^1 \int_{B_{2r(y)}} \eta \Bigg\{ b_{l}\Big[w^{\alpha + 1}|\nabla w|^{2l}S_{k - l}^{i j}  (u^{(m_2)} - u^{(m_1)})_{i}\Big]_{j} \\
			& +  2l b_{l}  \Big[w^{\alpha}|\nabla w|^{2l - 2}S^{i j}_{k - l + 1} w_{i}(u^{(m_2)} - u^{(m_1)}) \Big]_{j} \\
			&- (\alpha + n + l + 1)b_{l} \Big[w^{\alpha}|\nabla w|^{2l}S^{i j}_{k - l} w_{i}(u^{(m_2)} - u^{(m_1)}) \Big]_{j} \Bigg\}\\
			=&\sum_{l = 0}^k \int_0^1 \int_{B_{2r(y)}} - b_{l}  w^{\alpha + 1}|\nabla w|^{2l}S_{k - l}^{i j}  (u^{(m_2)} - u^{(m_1)})_{i} \eta_{j} \\
			&-  2l b_{l}   w^{\alpha}|\nabla w|^{2l - 2}S^{i j}_{k - l + 1} w_{i}(u^{(m_2)} - u^{(m_1)}) \eta_{j}\\
				& + (\alpha + n + l + 1)b_{l}  w^{\alpha}|\nabla w|^{2l}S^{i j}_{k - l} w_{i}(u^{(m_2)} - u^{(m_1)})  \eta_{j}\\
			=&\sum_{l = 0}^k \int_0^1 \int_{B_{2r(y)}}  b_{l} \Big[ w^{\alpha + 1}|\nabla w|^{2l}S_{k - l}^{i j}\eta_{j} \Big]_{i}  (u^{(m_2)} - u^{(m_1)}) \\
			&-  2l b_{l}   w^{\alpha}|\nabla w|^{2l - 2}S^{i j}_{k - l + 1} w_{i}(u^{(m_2)} - u^{(m_1)}) \eta_{j}\\
				& + (\alpha + n + l + 1)b_{l}  w^{\alpha}|\nabla w|^{2l}S^{i j}_{k - l} w_{i}(u^{(m_2)} - u^{(m_1)})  \eta_{j}.
		\end{align*}
		By Lemma \ref{2024.7.10lemma1}, we have
		\begin{equation}
			\left\vert\sum_{i}\partial_{i} (S_{k - l}^{i j})\right\vert \leq C w^{-1}|\nabla w| S_{k - l - 1}.
		\end{equation}
		Then by direct cpmputation , we find that
		\begin{equation}\label{2021.7.10.equ1}
			\begin{aligned}
				&\Big[ w^{\alpha + 1}|\nabla w|^{2l}S_{k - l}^{i j}\eta_{j} \Big]_{i}  \\
				=&(\alpha + 1)w^{\alpha}|\nabla w|^{2l}S_{k - l}^{i j}\eta_{j} w_{i} + 2l w^{\alpha + 1}|\nabla w|^{2l - 2}S_{k - l}^{i j}\eta_{j} w_{i m }w_{m} + w^{\alpha + 1}|\nabla w|^{2l}S_{k - l}^{i j}\eta_{ij}\\
				&+ w^{\alpha + 1}|\nabla w|^{2l}\eta_{j} \cdot \partial_{i} (S_{k - l}^{i j})\\
				=&(\alpha + 1)w^{\alpha}|\nabla w|^{2l}S_{k - l}^{i j}\eta_{j} w_{i} + 2l w^{\alpha }|\nabla w|^{2l - 2} S_{k - l}^{i j}\eta_{j}A_{im} w_{m} + l w^{\alpha }|\nabla w|^{2l} S_{k - l}^{i j}\eta_{j} w_{i}  + w^{\alpha + 1}|\nabla w|^{2l}S_{k - l}^{i j}\eta_{ij}\\
				&+ w^{\alpha + 1}|\nabla w|^{2l}\eta_{j} \cdot \partial_{i} (S_{k - l}^{i j})\\
				=&(\alpha + 1)w^{\alpha}|\nabla w|^{2l}S_{k - l}^{i j}\eta_{j} w_{i} + 2l w^{\alpha }|\nabla w|^{2l - 2}\eta_{j} w_{m}(S_{k - l}\delta_{mj} - S^{m j}_{k - l + 1}) + l w^{\alpha }|\nabla w|^{2l} S_{k - l}^{i j}\eta_{j} w_{i}  \\
				&+ w^{\alpha + 1}|\nabla w|^{2l}S_{k - l}^{i j}\eta_{ij} + w^{\alpha + 1}|\nabla w|^{2l}\eta_{j} \cdot \partial_{i} (S_{k - l}^{i j})\\
				\leq &Cw^{\alpha}|\nabla w|^{2l + 1}S_{k - l - 1} + Cl w^{\alpha }|\nabla w|^{2l - 1}S_{k - l}  + C w^{\alpha + 1}|\nabla w|^{2l}S_{k - l - 1},
			\end{aligned}
		\end{equation}
		and
		\begin{equation}
			 w^{\alpha}|\nabla w|^{2l}S^{i j}_{k - l} w_{i}  \eta_{j} \leq C w^{\alpha}|\nabla w|^{2l + 1}S_{k - l - 1} .
		\end{equation}
		So by Cauchy-Schwarz inequality, we get that
		\begin{align*}
			&\int \eta \psi^{\alpha}_{k}(u^{(m_1)}) - \eta \psi^{\alpha}_{k}(u^{(m_2)})\\
			 \leq&\sum_{l = 0}^{k - 1} \int_0^1 \int_{B_{2r(y)}} \Bigg[ C  w^{\alpha}|\nabla w|^{2l + 1}S_{k - l - 1}(u^{(m_2)} - u^{(m_1)})^+  + C w^{\alpha + 1}|\nabla w|^{2l}S_{k - l - 1}(u^{(m_2)} - u^{(m_1)})^{+}  \Bigg] \\
			 \leq&\sum_{l = 0}^{k - 1} \int_0^1 \int_{B_{2r(y)}} \Bigg[  C w^{\alpha + 1}|\nabla w|^{2l}S_{k - l - 1}  (u^{(m_2)} - u^{(m_1)})^{+}  + C  w^{\alpha - 1}|\nabla w|^{2l + 2}S_{k - l - 1}(u^{(m_2)} - u^{(m_1)})^+  \Bigg].
		\end{align*}
		
		Fix $\varepsilon \in (0, 1)$ and $N$ so that for 
		\begin{equation}
			O_{\varepsilon} := \{x \in B_{2r}(y)| |u^{(m_2)} - u^{(m_1)}|> \varepsilon \},
		\end{equation}
		we have $|O_{\varepsilon}| < \varepsilon$ if $m_1, m_2 \geq N$. Besides, we can suppose
		\begin{equation}
			\sup_{m} \int_{B_{3r}(y)} u^{(m)} \leq K,
		\end{equation}
		for some fixed constant K.
		 We then have
		\begin{align*}
			&\int \eta \psi^{\alpha}_{k}(u^{(m_1)}) - \eta \psi^{\alpha}_{k}(u^{(m_2)})\\
			 \leq&\sum_{l = 0}^{k - 1} \int_0^1 \int_{B_{2r(y)}}  C w^{\alpha + 1}|\nabla w|^{2l}S_{k - l - 1}  (u^{(m_2)} - u^{(m_1)} - \varepsilon)^{+}  +  C \varepsilon w^{\alpha + 1}|\nabla w|^{2l}S_{k - l - 1}\\
			 &+ C  w^{\alpha - 1}|\nabla w|^{2l + 2}S_{k - l - 1}(u^{(m_2)} - u^{(m_1)} - \varepsilon)^+ + C\varepsilon  w^{\alpha - 1}|\nabla w|^{2l + 2}S_{k - l - 1} . \\
		\end{align*}
		Using Lemma \ref{6.17bound}, we have
		\begin{align*}
			\int_0^1 \int_{B_{2r(y)}}w^{\alpha + 1}|\nabla w|^{2l}S_{k - l - 1} &\leq CK^{2k + \alpha - 1 } ,
		\end{align*}
		and
		\begin{align*}
			&\int_0^1 \int_{B_{2r(y)}}  w^{\alpha + 1}|\nabla w|^{2l}S_{k - l - 1}  (u^{(m_2)} - u^{(m_1)} - \varepsilon)^{+} \\
			\leq &\int_0^1 \int_{O_{\varepsilon}}  w^{\alpha + 1}|\nabla w|^{2l}S_{k - l - 1}  (u^{(m_2)} - u^{(m_1)} - \varepsilon)^{+} \\
			\leq &CK\int_0^1 \int_{O_{\varepsilon}}  w^{\alpha + 1}|\nabla w|^{2l}S_{k - l - 1} .  \\
		\end{align*}
		Then it follows that for $m_1, m_2 \geq N' \geq N$, for a further constant $N'$ depending on $\varepsilon, r$,
		\begin{equation}
			\int_0^1 \int_{B_{2r(y)}}  w^{\alpha + 1}|\nabla w|^{2l}S_{k - l - 1}  (u^{(m_2)} - u^{(m_1)} - \varepsilon)^{+} \leq C \varepsilon . 
		\end{equation}
		So far, we obtain that 
		\begin{equation}
			\int \eta \psi^{\alpha}_{k}(u^{(m_1)}) - \eta \psi^{\alpha}_{k}(u^{(m_2)}) \leq C\varepsilon.
		\end{equation}
		This means that $\psi^{\alpha}_k(u^{(m)})$ converges weakly to a Borel measure  in $\Omega$.
		\item Step2: 
		Note that 
		\begin{align*}
			&\int_{B_{2r(y)}} \eta (u^{(m_1)})^{\gamma} \psi^{\alpha}_{k}(u^{(m_1)}) - \eta (u^{(m_2)})^{\gamma} \psi^{\alpha}_{k}(u^{(m_2)})\\
			 =& \int_0^1 \int_{B_{2r(y)}} \eta \frac{\partial}{\partial t}\Big[w^{\gamma}\psi^{\alpha}_{k}(w) \Big]\\
			 =& \int_0^1 \int_{B_{2r(y)}} \eta \Big[w^{\gamma}\frac{\partial}{\partial t}\psi^{\alpha}_{k}(w) + \gamma w^{\gamma - 1}\psi^{\alpha}_{k}(w) \cdot(u^{(m_2)} - u^{(m_1)})\Big].\\
		\end{align*}
		For any $\gamma + \alpha + 2k > 1$, by the same arguments as Step1, we get $(u^{(m)})^{\gamma}\psi^{\alpha}_k(u^{(m)})$ converges weakly to a Borel measure  in $\Omega$. So we 
		get the weak convergence of the Borel measure  
		\begin{equation}
			\varphi_{k}^{\alpha}(u) : = \sum_{l = 0}^{k} b_{l} (\alpha)  |\nabla u|^{2l}S_{k - l}(u).
		\end{equation}
		Recall \eqref{6.19b}:
		\begin{equation}
			\begin{aligned}
				b_{l + 1} &=  \Big[(k + l + 1)\alpha + kn + 2k(l + 1) \Big]\cdot \frac{(\alpha + n + l)!}{2^{l + 1}(l + 1)! (\alpha + n)!},
			\end{aligned}
		\end{equation}
		with $0 \leq l \leq k - 1$ and $b_{0} = k.$ We find that for any $0 \leq l \leq k$, $b_{l} \approx C_{l} \alpha^{l}$ . So there always exist $\alpha_0, \alpha_1, \cdots , \alpha_{k} > 0$, such that 
		\begin{equation}
			\det (H_{i j}) \neq 0,
		\end{equation}
		where 
		\begin{equation}
			H_{i j} = b_j(\alpha_i), \text{ with $i, j = 0, \cdots, k$.}
		\end{equation}
		So for any $0 \leq l \leq k$, we get 
		the weak convergence of the Borel measure  
		\begin{equation}
			|\nabla u|^{2l}S_{k - l}(u).
		\end{equation}

	\end{itemize}

	\end{proof}

	By the above proof and combining Corollary \ref{coro1} , in fact we have already got:	
	\begin{corollary}
		For any $u \in \Phi^k(\Omega)$, $\alpha > 1 - 2k$ , $0 \leq l \leq k$ , there exists a Borel measure $\mu_{k, l, \alpha}[u]$ in $\Omega$ such that 
		\begin{itemize}
			\item $\mu_{k, l, \alpha}[u] = u^{\alpha} |\nabla u|^{2l}S_{k - l} (u)$ for $u \in C^2(\Omega)$.
			\item If $\{ u^{(m)} \}$ is a sequence in $\Phi^k(\Omega)$ converging locally in measure to a function $u \in \Phi^k(\Omega)$, then the corresponding measure $\mu_{k, l, \alpha}[u^{(m)}] \rightarrow \mu_{k, l, \alpha}[u]$ weakly.
		\end{itemize}
	\end{corollary}

\bibliography{measure}	
\bibliographystyle{plain}

\end{document}